\documentclass[12 pt]{article}%
\usepackage{amsmath, amsfonts, amsthm, color,latexsym}
\usepackage{amsmath, ulem}
\usepackage{amsfonts}
\usepackage{amssymb}
\usepackage{color, soul}
\usepackage[all]{xy}
\usepackage{graphicx}%
\providecommand{\U}[1]{\protect\rule{.1in}{.1in}}
\allowdisplaybreaks[4]
\newtheorem{theorem}{Theorem}[section]
\newtheorem{proposition}[theorem]{Proposition}
\newtheorem{corollary}[theorem]{Corollary}
\newtheorem{example}[theorem]{Example}

\newtheorem{remark}[theorem]{Remark}

\newtheorem{lemma}[theorem]{Lemma}
\newtheorem{final remark}[theorem]{Final Remark}

\allowdisplaybreaks[4]

\newcommand{\sol}[1]{\text{sol}(#1)}

\newcommand {\cvfe} {\overset{\omega^\ast}{\rightarrow}}

\newcommand {\R}{\mathbb{R}}

\newcommand {\N} {\mathbb{N}}
\newcommand{\A}{\mathcal{A}}
\newcommand{\norma}[1]{\| #1 \|}
\newcommand{\conj}[2]{\left \{ {#1} \, : \, {#2} \right \}}

\begin{document}

\title{Banach lattices of linear operators and of homogeneous polynomials not containing $c_0$}
\author{Geraldo Botelho\thanks{Supported by Fapemig grants RED-00133-21 and APQ-01853-23.}~,~Vinícius C. C. Miranda\thanks{Supported by CNPq Grant 150894/2022-8 and Fapemig grant APQ-01853-23\newline 2020 Mathematics Subject Classification: 46B42, 46B28, 46G25.\newline Keywords: Banach lattices, regular homogeneous polynomials, compact polynomials, weakly compact polynomials, almost limited polynomials, $p$-compact polynomials. }~~and Pilar Rueda}
\date{}
\maketitle

\begin{abstract}
\noindent First we develop a technique to construct Banach lattices of homogeneous polynomials. We obtain, in particular, conditions for 
the linear spans of all positive compact and weakly compact $n$-homogeneous polynomials between the Banach lattices $E$ and $F$, denoted by 
${\cal P}_{\cal K}^r(^n E; F)$ and $\mathcal{P}_{\mathcal{W}}^r(^n E; F)$, to be Banach lattices with the polynomial regular norm. 
Next we study when the following are equivalent for ${\cal I} = {\cal K}$ or ${\cal I} = {\cal W}$: (1) The space $\mathcal{P}^r(^n E; F)$ of regular polynomials contains no copy of $c_0$. (2) ${\cal P}_{\mathcal{I}}^r(^n E; F)$ contains no copy of $c_0$.  (3) ${\cal P}_{\mathcal{I}}^r(^n E; F)$ is a projection band in $\mathcal{P}^r(^n E; F)$. 
  (4) Every positive polynomial in $\mathcal{P}^r(^n E; F)$ belongs to ${\cal P}_{\cal I}^r(^nE;F)$. The result we obtain in the compact case can be regarded as a lattice polynomial Kalton theorem. Most of our results and examples are new even in the linear case $n = 1$. 
\end{abstract}

\section{Introduction}

A classical problem in Functional Analysis consists in studying embeddability of $c_0$ in spaces of bounded linear operators between Banach spaces (see, e.g., \cite{bator,  emmanuele2, emmanuele1, kalton}). One of the most known results in this direction is Kalton's theorem \cite[Theorem 6]{kalton} which states that
for a Banach space $X$ with an unconditional finite-dimensional expansion of the identity and an infinite dimensional Banach space $Y$, the space $\mathcal{L}(X;Y)$ of all bounded linear operators from $X$ to $Y$ contains no copy of $c_0$ if and only if every bounded linear operator from $X$ to $Y$ is compact. In \cite{perez}, S. Pérez studied the embeddability of $c_0$ in the space $\mathcal{P}(^n X; Y)$ of all continuous $n$-homogeneous polynomials. In the lattice setting, this issue is specially important because the non embeddability of $c_0$ in a Banach lattice is equivalent to the lattice being a $KB$-space. In this direction, extending previous results from \cite{ bulixue, buwong, lijibu}, F. Xanthos \cite{xanthos} gave the following version of Kalton's theorem for the Banach lattice $\mathcal{L}^r(^n E; F)$ of all regular linear operators:
for an atomic Banach lattice $E$ with order continuous norm and an arbitrary Banach lattice $F$, $\mathcal{L}^r(^n E; F)$ contains no copy of $c_0$ if and only if every positive linear operator from $E$ to $F$ is compact (see \cite[Theorem 2.9]{xanthos}).

The interest in studying polynomial versions of well known results or properties in Banach lattice theory have been considerably increased recently (see \cite{botelholuiz, botlumir,  boydryansnig2, boydryansnig, buquaterly, galmir1, libu, wangshibu}). It is then a natural question to seek for a ``lattice polynomial version" of Kalton's theorem \cite[Theorem 6]{kalton}. The main purpose of this manuscript is to obtain conditions on the Banach lattice $E$ and $F$ so that the Banach lattice $\mathcal{P}^r(^n E; F)$ of all regular $n$-homogeneous polynomials from $E$  to $F$  contains no copy of $c_0$ if and only if every positive $n$-homogeneous polynomial from $E$ to $F$ is compact.
In order to achieve this result, we need to introduce a complete lattice norm on the space $\mathcal{P}_{\mathcal{K}}^r(^n E; F)$, which is the linear span of all positive compact $n$-homogeneous polynomials from $E$ to $F$. In general, $\mathcal{P}_{\mathcal{K}}^r(^n E; F)$ is not a sublattice of $\mathcal{P}^r(^n E; F)$. Indeed, for $n = 1$, examples where $\mathcal{K}^r(E; F) := \mathcal{P}_{\mathcal{K}}^r(^1 E; F)$ is not a sublattice of $\mathcal{L}^r(^n E; F)$ and $\mathcal{K}^r(E; F)$ is not closed in $\mathcal{L}^r(^n E; F)$ with the regular norm are well known (see \cite[Corollary 3]{abrawick} and \cite[Corollaries 2.9 and 2.10]{chenwick2}). Fortunately, there exists  a norm (called $k$-norm) on $\mathcal{K}^r(E; F)$ under which it is a Banach space (see \cite[Proposition 2.2]{chenwick2}).
In Section $2$, we are going to give conditions such that $\mathcal{P}_{\mathcal{K}}^r(^n E; F)$ is a Banach lattice with a certain norm which  will coincide with the regular norm in this case (Example \ref{corcompacto}). Actually, we will prove a more general result showing how to construct Banach lattices of regular linear operators and of regular homogeneous polynomials (Theorem \ref{teogeral}). The case of compact polynomials, along with other interesting classes of polynomials  and linear operators, follow,  in Sections 2 and 3, as particular instances of the general construction. 

In Section 4, we will prove that for a Banach lattice $E$ that fails the dual positive Schur property and an infinite dimensional atomic Banach lattice $F$ with order continuous norm, $\mathcal{P}^r(^n E; F)$ contains no copy of $c_0$ if and only if every positive $n$-homogeneous polynomial from $E$ to $F$ is compact. 
In the weakly compact case we write $\mathcal{P}_{\mathcal{W}}^r(^n E; F)$ for the linear span of all positive weakly compact $n$-homogeneous polynomials from $E$ to $F$. A result similar to the compact case shall be obtained for $E$ failing the positive Grothendieck property, $F$ atomic Dedekind complete and $\mathcal{P}^r(^n E; F)$ with order continuous norm.

To the best of our knowledge, except for part of Example \ref{corcompacto}, our results and examples are new even in the linear case $n = 1$.


We refer the reader to \cite{alip, meyer} for background on Banach lattices, to \cite{fabian} for Banach space theory,  and to \cite{dineen} for polynomials on Banach spaces. Throughout this paper, $E$ and $F$ denote Banach lattices, $E^+$, $B_E$ and $S_E$ denote, respectively, the positive  cone, the closed unit ball and the unit sphere of $E$, and, whenever $P$ is a continuous homogeneous polynomial, $\|P\|$ denotes its usual sup norm, that is, $\|P\| = \sup\{\|P(x)\|: x \in B_E\}$. For a subset $A \subset E$, we denote by $\text{co}(A)$ its convex hull, by $\text{sol}(A)$ its solid hull and by $\text{sco}(A)$ its solid convex hull. The symbol $E \cong F$ means that there exists an isometric isomorphism from $E$ to $F$ that is also a lattice homomorphism.

\section{Banach lattices of operators and of polynomials}

We begin by recalling the terminology and a few properties concerning regular polynomials and the positive projective symmetric tensor product. Details can be found in \cite{bubuskes} and \cite{loane}.

A $n$-homogeneous polynomial between Riesz spaces $P \colon E \to F$ is positive  if its associated symmetric multilinear operator $T_P\colon E^n \to F$ is positive, meaning that $T_P(x_1, \ldots, x_n) \geq 0$ for all $x_j \in E_j^+, j = 1,\ldots, n$. The difference of two positive $n$-homogeneous polynomials is called a regular homogeneous polynomial, and the set of all these polynomials is denoted by $\mathcal{P}^r(^nE, F)$. When $F$ is the scalar field we simply write $\mathcal{P}^r(^nE)$. If $E$ and $F$ are Banach lattices with $F$ Dedekind complete, then $\mathcal{P}^r(^nE, F)$ is a Banach lattice with the regular norm
$$\norma{P}_r = \norma{|P|} = \inf\{ \|Q\| : Q \in {\cal P}^+(^n E;F), \,Q \geq |P|\},$$ where $|P|$ denotes the absolute value of the regular $n$-homogeneous polynomial $P \colon E \to F$.

For a Banach lattice $E$, we denote the
$n$-fold positive projective symmetric tensor product of $E$ by $\widehat{\otimes}_{s, |\pi|}^n E$, which is a Banach lattice
endowed with the positive projective symmetric tensor norm $\norma{\cdot}_{s, |\pi|}$. As usual, we write $\otimes^n x = x \, \otimes \stackrel{n}{\cdots} \otimes \, x$ for every $x \in E$, and for a subset $A$, we write $\otimes^n[A] = \conj{\otimes^n x}{x \in A}$.  
For every $P \in \mathcal{P}^r(^nE, F)$ there exists a unique regular linear operator $P^\otimes \colon \widehat{\otimes}_{s, |\pi|}^n E \to  F$, called the linearization of $P$, such that $P(x) = P^\otimes(\otimes^n x)$ for every $x \in E$. Moreover, the correspondence
\begin{equation}P \in \mathcal{P}^r(^n E, F) \mapsto P^\otimes \in \mathcal{L}^r\left(\widehat{\otimes}_{s, |\pi|}^n E; F\right)  \label{omku}\end{equation}
is an isometric isomorphism and a lattice homomorphism. 

If $\mathcal{A}$ is a vector subspace of $\mathcal{P}(^n E; F)$, we denote by $\mathcal{A}^+$ the class of all positive $n$-homogeneous polynomials belonging to $\mathcal{A}$. We say that the ordered pair $(E, F)$ satisfies the $\mathcal{A}$-domination property if, for all  positive $n$-homogeneous polynomials $P, Q \colon E \to F$ with $0 \leq P \leq Q \in \mathcal{A}$, it holds $P \in \A$.

\begin{theorem} \label{teogeral}
    Let $E$ and $F$ be Banach lattices with $F$ Dedekind complete and let $\A$ be a subspace of $\mathcal{P}(^n E; F)$ endowed with a complete norm $\norma{\cdot}_{\A}$ satisfying the following conditions: \\
    {\rm (I)} $\norma{P} \leq \norma{P}_\A$ for every $P \in \A$. \\
    {\rm (II)} $\norma{P}_{\A} \leq \norma{Q}_\A$ for all $P \in \A$ and $Q \in \A^+$ with  $|P(x)| \leq Q(|x|)$ for every $x \in E$.\\
    Then
     $$ \norma{P}_{\mathcal{A},r} := \inf \conj{\norma{Q}_\A}{Q \in \mathcal{A}^+, Q \geq |P|} $$
     defines a complete norm on $\mathcal{A}^r := \text{span}\{\mathcal{A}^+ \}$
    such that $\norma{P}_{\A,r} \geq \norma{P}_r$ for every $P\in \A^r$.  If, in addition, $(E, F)$ satisfies the $\mathcal{A}$-domination property, 
    then $(\mathcal{A}^r, \norma{\cdot}_{\A,r})$ is a Banach lattice and $\|\cdot\|_{\cal A} \leq \|\cdot\|_{{\cal A},r}$. Moreover, in this case, $\A^r$ is an ideal in $\mathcal{P}^r(^n E; F)$.
\end{theorem}

\begin{proof}     Observe first that, since $F$ is Dedekind complete, $\mathcal{P}^r(^n E; F)$ is a vector lattice, so we can consider $|P|$ for every $P \in \mathcal{P}^r(^n E; F)$. Besides, since $\A^+ \subset \mathcal{P}^+(^n E; F)$ and $\norma{P} \leq \norma{P}_\A$ for every $P \in \A$, we have
    \begin{align*}
        \norma{P}_{\A,r} & = \inf \conj{\norma{Q}_\A}{Q \in \mathcal{A}^+,\, Q \geq |P|} \\
            & \geq \inf \conj{\norma{Q}}{Q \in {\cal A}^+,\, Q \geq |P|} \\
            & \geq \inf \conj{\norma{Q}}{Q \in \mathcal{P}^+(^n E; F), \,Q \geq |P|} = \norma{P}_r
    \end{align*}
    for every $P \in \A^r$.
    Let us prove that $\norma{\cdot}_{\A ,r}$ is a norm in $\A^r$: \\
    {\rm (i)} If $\norma{P}_{\A ,r} = 0$, there exists a sequence $(Q_i)_i \subset \A^+$ with $Q_i \geq |P|$ for any $i \in \N$ and $\lim\limits_{i \to \infty} \norma{Q_i}_{\A} = 0$. By condition (I) we get
     $\lim\limits_{i \to \infty} Q_i(x) = 0$ for every $x \in E$. Hence,    $$ |P(x)| \leq |P|(|x|) \leq \lim_{i \to \infty} Q_i(|x|) = 0 $$
    for every $x \in E$, which implies that $P = 0$. \\
    {\rm (ii)} Let $\lambda \in \R$ and $P \in \A^r$ be given. For $\lambda = 0$, it is immediate that $\norma{\lambda P}_{\A, r} = 0 = |\lambda|\!\cdot\! \norma{P}_{\A, r}$. Assume now that $\lambda \neq 0$.
    On the one hand, given $Q \in \A^+$ with $Q \geq |P|$, we have that $|\lambda| Q \colon E \to F$ is a positive $n$-homogeneous polynomial belonging to $\A$ and satisfying $|\lambda| Q \geq |\lambda P|$, which implies that $\norma{\lambda P}_{\A,r} \leq \norma{|\lambda| Q}_\A = |\lambda| \!\cdot\!\norma{Q}_\A$. It follows that $\norma{\lambda P}_{\A,r} \leq |\lambda|\!\cdot\! \norma{P}_{\A,r}$.

     On the other hand, given $Q \in \A^+$ with $Q \geq |\lambda P|$, we have that $|P| \leq \frac{1}{|\lambda|} Q \in  \A^+$. Then, $\norma{P}_{\A,r} \leq \norma{\frac{1}{|\lambda|} Q}_\A = \frac{1}{|\lambda|}\!\cdot\! \norma{Q}_\A$, therefore $|\lambda|\!\cdot\!\norma{P}_{\A,r} \leq \norma{\lambda P}_{\A,r}$. \\
    {\rm (iii)} Let $P_1, P_2 \in \A^r$ be given.
    Letting $Q_1, Q_2 \colon E \to F$ be two positive $n$-homogeneous polynomials belonging to $\mathcal{A}$ with $Q_i \geq |P_i|$, we have that $Q_1 + Q_2 \in \mathcal{A}^+$ is such that $Q_1 + Q_2 \geq |P_1| + |P_2| \geq |P_1 + P_2|$, which yields that
    $$ \norma{P_1 + P_2}_{\A,r} \leq \norma{Q_1 + Q_2}_\A \leq \norma{Q_1}_\A + \norma{Q_2}_\A. $$
    Fixing $Q_1$, we get
    $$\norma{P_1 + P_2}_{\A,r} - \norma{Q_1}_\A \leq \inf \conj{\norma{Q}_\A}{Q \in \A^+, \, Q \geq |P_2|} = \norma{P_2}_{\A,r}.$$
    Hence,
    $$ \norma{P_1 + P_2}_{\A,r} - \norma{P_2}_{\A,r} \leq \norma{Q_1}_\A $$
    for every positive $n$-homogeneous polynomial $Q_1 \in {\cal A}$ satisfying $Q_1 \geq |P_1|$, so
    $$ \norma{P_1 + P_2}_{\A,r} \leq \norma{P_1}_{\A,r} + \norma{P_2}_{\A,r}. $$

    Suppose now that the pair $(E,F)$ satisfies the $\A$-domination property. Since $\mathcal{P}^r(^n E; F)$ is a Riesz space, to prove that $\A^r$ is a vector sublattice of $\mathcal{P}^r(^n E; F)$, it suffices us to check that $P^+ \in \A$ for every $P \in \A^r$. Let $P = P_1 - P_2 \in \A^r$ with $P_1, P_2 \in \A^+$. As $P_1 \geq 0$ and $P_1 \geq P$, we have $P_1 \geq P^+$, and by assumption we get $P^+ \in \A$. This proves that $\A^r$ is a vector lattice. We claim that $\norma{\cdot}_{\A, r}$ is a lattice norm on $\A^r$. Indeed, let $P, Q \in \A^r$ be such that $|P| \leq |Q|$. If $R \in \A^+$ is such that $R \geq |Q|$, then $R \geq |P|$, which implies by condition (II) of the assumptions that $\norma{|P|}_{\A} \leq \norma{R}_{\A}$. Thus $$\norma{|P|}_{\A} \leq \inf \conj{\norma{R}_{\A}}{R \in \A^+, \, R \geq |Q|} = \norma{Q}_{\A,r}.$$
    Since $\norma{P}_{\A,r} \leq \norma{|P|}_{\A}$, we have $\norma{P}_{\A,r} \leq \norma{Q}_{\A,r}$, which proves that $(\A^r, \norma{\cdot}_{\A, r})$ is a normed Riesz space.

     To prove that $(\mathcal{A}^r, \norma{\cdot}_{\A,r})$ is complete, let $(P_i)_i \subset \mathcal{A}^r$ be a $\norma{\cdot}_{\A,r}$-Cauchy sequence. The proof is similar as in \cite[Proposition 1.3.6]{meyer}. By passing to a subsequence if necessary, we may assume that
    $$\norma{P_i - P_{i+1}}_{\A,r} = \inf \conj{\norma{Q}_\A}{Q \in \A^+, \, Q \geq |P_i - P_{i+1}|} < 2^{-i}$$
    for any $i \in \N$. For every $i \in \N$, let $Q_i \in \A^+$ be such that $Q_i \geq |P_i - P_{i+1}|$ and $\norma{Q_i}_\A < 2^{-i}$. Since $(\A, \norma{\cdot}_\A)$ is a Banach space, for each $i \in \N$ there exists $R_i = \sum\limits_{j=i}^\infty Q_j \in \A$ with $\norma{R_i}_\A \leq \sum\limits_{j=i}^\infty \norma{Q_j}_\A < 2^{1-i}$. Moreover, $R_i \geq 0$ for every $i \in \N$.
     On the other hand, condition (II) and the fact that $\A^r$ is a vector lattice give $\norma{\cdot}_\A \leq \norma{\cdot}_{\A,r}$, so $(P_i)$ is a Cauchy sequence in the Banach space $(\A, \norma{\cdot}_\A)$. Let $P 
     \in \A$ be the $\|\cdot\|_{\cal A}$-limit of the sequence $(P_i)_i$. For each $i \in \N$ and each $x \in E$, we have
     \begin{align*}
         |(P - P_i)(x)| & = \lim_{ j \to \infty} |(P_j - P_i)(x)| \leq \lim_{j \to \infty} \sum_{k=i}^{j-1} |(P_{k+1} - P_k)(x)| \\
                        & \leq \lim_{j \to \infty} \sum_{k=i}^{j-1} |P_{k+1} - P_k|(|x|)  \leq \lim_{j \to \infty} \sum_{k=i}^{j-1} Q_k(|x|) \leq R_i(|x|),
     \end{align*}
     which implies that $P - P_i \leq R_i$ for every $i \in \N$. So, $P - P_i= R_i - (R_i - (P - P_i)) \in \A^r$, consequently $P \in \mathcal{A}^r$. Moreover, we have that $|P-P_i| \in \A^r$ (because $\A^r$ is a vector lattice), $R_i \in \A^+$ and $|P - P_i| \leq R_i$ for every $i \in \N$, hence
     $$ \norma{P - P_i}_{\A,r} \leq \norma{R_i}_{\A} \leq 2^{1-i} \longrightarrow 0 \mbox{ as } i \to \infty.$$
This proves that $(\mathcal{A}^r, \norma{\cdot}_{\A,r})$ is complete. It remains to check that $\A^r$ is an ideal in $\mathcal{P}^r(^n E; F)$.
    Taking $P \in \A^r$ and $Q \in \mathcal{P}^r(^n E; F)$ with $0 \leq |Q| \leq |P|$, since $|P|\in \A^r$ and the pair $(E, F)$ has the $\A$-domination property, we get $|Q| \in \A^r$. Moreover, from $0 \leq Q^+, Q^- \leq |Q|$ and from $\A$-domination property of the pair $(E, F)$ it follows that $Q^+, Q^- \in \A^r$, therefore $Q \in \A^r$.
\end{proof}

\begin{corollary}\label{qnh2}
       Let $E$ and $F$ be Banach lattices with $F$ Dedekind complete. If $\mathcal{A}$ is a closed subspace of $\mathcal{P}(^n E; F)$ such that $(E, F)$ satisfies the $\mathcal{A}$-domination property, then
     $$ \norma{P}_{\mathcal{A},r} := \inf \conj{\norma{Q}}{Q \in \mathcal{A}^+, Q \geq |P|} $$
     defines a complete lattice norm on $\mathcal{A}^r = \text{span}\{\mathcal{A}^+ \}$, that is,
   $(\mathcal{A}^r, \norma{\cdot}_{\A,r})$ is a Banach lattice. Moreover, $\norma{P}_{\A,r} = \norma{P}_{r}$ for every $P \in \mathcal{A}^r$ and $\A^r$ is an ideal in $\mathcal{P}^r(^n E; F)$.
\end{corollary}

\begin{proof} Since the usual sup norm $\|\cdot\|$ enjoys conditions (I) and (II), from Theorem \ref{teogeral} it follows that 
$(\mathcal{A}^r, \norma{\cdot}_{\A,r})$ is a Banach lattice and $\norma{P}_{\A, r} \geq \norma{P}_r$ for every $P \in \mathcal{A}^r$. To see the converse inequality, note that
    $ \norma{P}_r = \norma{|P|} \geq \norma{P}_{\A, r} $ holds for every $P \in \mathcal{A}^r$.
\end{proof}

 The next two examples follow immediately from the results proved thus far.

\begin{example}  \label{corcompacto}\rm
    Given Banach lattices $E$ and $F$ with $F$ Dedekind complete, let $\mathcal{P}_{\mathcal{K}}(^n E; F)$ denote the closed subspace of $\mathcal{P}(^n E; F)$ of all compact $n$-homogeneous polynomials from $E$ to $F$. By taking $\A = \mathcal{P}_{\mathcal{K}}(^n E; F)$ in Theorem \ref{teogeral}, we obtain that
    $$\norma{P}_{{\cal K},r} = \inf \conj{\norma{Q}}{Q \in \mathcal{P}_{\mathcal{K}}^+ (^n E; F), \, Q \geq |P|} $$
    defines a complete norm on $\mathcal{P}_{\mathcal{K}}^r(^n E; F)$.
    In addition, if $F$ is atomic with order continuous norm or an AL-space, then the pair $(E,F)$ satisfies the domination property for compact homogeneous polynomials (see \cite[Corollary 4.2 or Corollary 4.3]{libu}). By Theorem \ref{teogeral} we obtain that $(\mathcal{P}_{\mathcal{K}}^r(^n E; F), \norma{\cdot}_{{\cal K},r} )$ is a Banach lattice such that $\norma{P}_{{\cal K},r} = \norma{P}_r$ for every $P \in \mathcal{P}_{\mathcal{K}}^r(^n E; F)$. Moreover, under the condition on $F$ being atomic with order continuous norm or an AL-space, writing  $\mathcal{K}^r(\widehat{\otimes}_{s, |\pi|}^n E; F) := \mathcal{P}_{\mathcal{K}}^r(^1 \widehat{\otimes}_{s, |\pi|}^n E; F)$, the correspondence
    $$ P \in  \mathcal{P}_{\mathcal{K}}^r(^n E; F) \mapsto P^\otimes \in  \mathcal{K}^r(\widehat{\otimes}_{s, |\pi|}^n E; F)$$
    is an  isometric isomorphism and a lattice homomorphism. Indeed, since the correspondence (\ref{omku})
    is an isometric isomorphism and a lattice homomorphism,
     it is enough to check that, for a given regular $n$-homogeneous polynomial $P \colon E \to F$, $P \in \mathcal{P}_{\mathcal{K}}^r(^n E; F)$ if and only if $P^\otimes \in {\mathcal{K}}^r(^n E; F)$.  This follows immediately from \cite[Theorem 4.1 or Theorem 4.3]{libu} and from the fact that both $\mathcal{P}_{\mathcal{K}} ^r(^n E; F)$ and $\mathcal{K}^r(\widehat{\otimes}_{s, |\pi|}^n E; F)$ are Banach lattices.

    For $F$ atomic with order continuous norm, the linear case of this example was obtained in \cite[Theorem 5.5]{1998}.
\end{example}

\begin{example} \label{corfraccompacto}\rm
    Given Banach lattices $E$ and $F$ with $F$ Dedekind complete, let $\mathcal{P}_{\mathcal{W}}(^n E; F)$ denote the closed subspace of $\mathcal{P}(^n E; F)$ of all weakly compact $n$-homogeneous polynomials from $E$ to $F$. By taking $\A = \mathcal{P}_{\mathcal{W}}(^n E; F)$ in Theorem \ref{teogeral}, we obtain that
    $$ \norma{P}_{{\cal W},r} = \inf \conj{\norma{Q}}{Q \in \mathcal{P}_{\mathcal{W}}^+ (^n E; F), \, Q \geq |P|} $$
    defines a complete norm on $\mathcal{P}_{\mathcal{W}}^r(^n E; F)$. In addition, if $F$ has order continuous norm, then the pair $(E, F)$ satisfies the domination property for weakly compact homogeneous polynomials (see \cite[Corollary 3.3]{libu}). By Theorem \ref{teogeral} we obtain that $(\mathcal{P}_{\mathcal{W}}^r (^n E; F), \norma{\cdot}_{{\cal W},r})$ is a Banach lattice such that $\norma{P}_{{\cal W},r} = \norma{P}_r$ for every $P \in \mathcal{P}_{\mathcal{W}}^r (^n E; F)$. Moreover, under the condition of $F$ having order continuous norm, reasoning as in Example \ref{corcompacto} and     applying \cite[Theorem 3.2]{libu}, we have that the correspondence
    $$ P \in  \mathcal{P}_{\mathcal{W}}^r(^n E; F) \mapsto P^\otimes \in  \mathcal{W}^r(\widehat{\otimes}_{s, |\pi|}^n E; F):=\mathcal{P}_{\mathcal{W}}^r(^1 \widehat{\otimes}_{s, |\pi|}^n E; F)$$
    is an isometric isomorphism and a lattice homomorphism.
\end{example}

\section{Almost limited and $p$-compact polynomials}

In this section we give two more applications of the results proved in the previous section for classes of polynomials other than compact and weakly compact. We will consider almost limited polynomials and solid $p$-compact polynomials. The difference from the compact and weakly compact cases is that the results we need for these further classes of polynomials are not available in the literature. In particular, answers to the following questions are unknown:

$\bullet$ Under which conditions has a pair of Banach lattices $(E,F)$ the domination property for these classes?

$\bullet$ Is it true that a positive homogeneous polynomial belongs to the class if and only if its linearization does?

We start by considering the class of almost limited polynomials. Extending the linear notion from \cite{elbour} to the polynomial case, we say that an $n$-homogeneous polynomial $P \colon E \to F$ is said to be {\it almost limited} if $P(B_E)$ is an almost limited subset of $F$, that is, for every disjoint weak* null sequence $(y_n^\ast)_n \subset F^\ast$, $\norma{y_n^\ast \circ P} = \sup\limits_{x \in B_E} y_n^\ast(P(x)) \to 0$. We give a brief proof that the class $\mathcal{P}_{\rm al}(^n E; F)$ of all almost limited $n$-homogeneous polynomials $P \colon E \to F$ is a closed subspace of $\mathcal{P}(^n E; F)$. Let $(P_i)_i \subset \mathcal{P}_{\rm al}(^n E; F)$ be a sequence such that $\norma{P_i - P} \to 0$, let $(y_j^\ast)_j$ be a disjoint weak* null sequence and let $\varepsilon > 0$ be given. Since $\norma{P_i - P} \to 0$, there exists $i_0 \in \N$ such that
    $\norma{P_{i_0} - P} < \frac{\varepsilon}{2\sup_{j \in \N} \norma{y_j^\ast}}.$ Moreover, since $P_{i_0}$ is an almost limited polynomial, there exists $j_0 \in \N$ such that $\norma{y_j^\ast(P_{i_0})} < \varepsilon/2$ for every $j > j_0$. Hence
    \begin{align*}
        \sup_{x \in B_E} |y_j^\ast (P(x))| & = \norma{y_j^\ast \circ P}  \leq \norma{y_j^\ast \circ P - y_j^\ast \circ P_{i_0}} + \norma{y_j^\ast \circ P_{i_0}} \\
        & \leq \norma{y_j^\ast} \!\cdot\!\norma{P - P_{i_0}} + \norma{y_j^\ast(P_{i_0})} < \varepsilon
    \end{align*}
for every $j > j_0$, which proves that $P$ is almost limited.

The following lemma will be needed in the study of the domination problem for almost limited polynomials.

\begin{lemma} \label{lemaconv}
    If $A$ is an almost limited subset of $E$, then its convex hull $\text{co}(A)$ is also an almost limited set.
\end{lemma}

\begin{proof}
      We recall from \cite[Proposition 2.3]{galmir1} that a subset $A$ of a Banach lattice $E$ is almost limited if and only if $T(A)$ is a relatively compact subset of $c_0$ for every disjoint operator $T \colon E \to c_0$ (i.e. the weak* null sequence $(x_n^\ast)_n \subset E^\ast$  that defines $T$ is disjoint). If $T \colon E \to c_0$ is a disjoint operator, then $T(A)$ is a relatively compact subset of $c_0$, thus $\text{co}(T(A))$ is also a relatively compact subset of $c_0$ \cite[Theorem 3.4]{alip}. Since $T(\text{co}(A)) \subset \text{co}(T(A))$, we conclude that $T(\text{co}(A))$ is relatively compact, hence $\text{co}(A)$ is almost limited.
\end{proof}

In order to solve the linearization problem and the domination problem for almost limited polynomials, we recall that a Banach lattice $E$ is said to have {\it property (d)} if $|x_n^\ast| \cvfe 0$  in $E^\ast$ for every disjoint weak* null sequence $(x_n^\ast)_n$ (see, e.g., \cite[Definition 1]{elbour}).

\begin{theorem} \label{domalmostlimited}
    Let $E$ and $F$ be two Banach lattices with $F$ having property (d). Then, a positive $n$-homogeneous polynomial $P \colon E \to F$ is almost limited if and only if its linearization $P^\otimes$ is an almost limited operator. Besides, if $P \colon E \to F$ is an almost limited  positive $n$-homogeneous polynomial, then $[0, P]$ is contained in the class of the almost limited polynomials.
\end{theorem}

\begin{proof}
    If $P \colon E \to F$ is an almost limited positive $n$-homogeneous polynomial, then $P(B_E^+)$ is an almost limited subset of $F$, so $\text{co}(P(B_E^+))$ is also an almost limited subset of $F$ by Lemma \ref{lemaconv}. Since $F$ has property (d), we obtain from \cite[Proposition 2.2]{loumir1} that $\text{sco}(P(B_E^+)) = \text{sol}(\text{co}(P(B_E^+)))$ is also an almost limited set. Since $ P^\otimes (B_{\hat{\otimes}_{s, |\pi|}^n E})\subset \text{sco}(P(B_E^+))$ (see \cite[Lemma 3.1]{libu}), we get that $P^\otimes (B_{\hat{\otimes}_{s, |\pi|}^n E})$ is an almost limited set, which implies that $P^\otimes$ is an almost limited operator.
    Conversely, assume that $P^\otimes \colon \hat{\otimes}_{s, |\pi|}^n E \to F$ is an almost limited operator. To prove that $P(B_E)$ is an almost limited subset of $F$, let $(y_n^\ast)_n$ be a disjoint weak* null sequence in $F^\ast$. As $P^\otimes (B_{\hat{\otimes}_{s, |\pi|}^n E})$ is an almost limited subset of $F$ and $\otimes^n[B_E] \subset B_{\hat{\otimes}_{s, |\pi|}^n E}$, we have that
    $$ \sup_{x \in B_E} |y_n^\ast (P(x))| = \sup_{x \in B_E} |y_n^\ast (P^\otimes (\otimes^n(x)))| \leq \sup_{z \in B_{\hat{\otimes}_{s, |\pi|}^n E}} |y_n^\ast (P^\otimes (z))| \to 0, $$
    which yields that $P(B_E)$ is almost limited set, proving that $P$ is an almost limited polynomial.

    Assume now that $P \colon E \to F$ is an almost limited positive $n$-homogeneous polynomial. If $Q \colon E \to F$ is a positive $n$-homogeneous polynomial with $Q \leq P$, then $0 \leq Q^\otimes \leq P^\otimes$. Since $F$ has property (d) and $P^\otimes$ is an almost limited operator, we get from \cite[Corollary 3]{elbour} that $Q^\otimes$ is an almost limited operator, hence $Q$ is an almost limited polynomial.
\end{proof}

Now we can apply Corollary \ref{qnh2} to the class of almost limited polynomials.

\begin{example} \label{coralmostlimited} \rm
Let $E$ and $F$ be two Banach lattices with $F$ Dedekind complete.  Since Dedekind complete Banach lattices have property (d), we get from Theorem \ref{domalmostlimited} that the pair $(E,F)$ satisfies the domination property for almost limited homogeneous polynomials. By taking $\A = \mathcal{P}_{\rm al}(^n E; F)$  in Corollary \ref{teogeral} we obtain that $\mathcal{P}_{\rm al}^r(^n E; F)$ is a Banach lattice with the norm
$$ \norma{P}_{{\rm al},r} = \inf \conj{\norma{Q}}{Q \in \mathcal{P}_{\rm al}^+(^n E; F), \, Q \geq |P|}. $$
Moreover, 
$\norma{P}_{{\rm al},r} = \norma{P}_r$ for every $P \in \mathcal{P}_{\rm al}^r(^n E; F)$. Following the same argument from Example \ref{corcompacto}, we obtain from Theorem \ref{domalmostlimited} that  the correspondence $$ P \in \mathcal{P}_{\rm al}^r(^n E; F) \mapsto P^\otimes \in  \mathcal{L}_{\rm al}^r(\hat{\otimes}_{s, |\pi|}^n E; F)$$
    is an isomorphism between Banach lattices, where $\mathcal{L}_{\rm al}^r(E; F) := \mathcal{P}_{\rm al}^r(^1 E; F)$.
\end{example}

The second class of polynomials we consider in this section is a lattice counterpart of the well studied class of $p$-compact polynomials, see, e.g. \cite{aronracsam, aronruedacontemp, aronruedaprims, ewertonlama}. Given $1 \leq p < \infty$, let $p^*$ be given by $\frac{1}{p} + \frac{1}{p^*} = 1$. For a Banach space $E$, we denote by $\ell_p(E)$ the space of absolutely $p$-summable $E$-valued sequences endowed with its usual norm $\|\cdot\|_p$, and by $\ell_p^w(E)$ the space of weakly $p$-summable $E$-valued sequences. According to \cite{sinhakarn}, a subset $K$ of a Banach space $E$ is {\it relatively $p$-compact} if there is a sequence $(x_j)_j \in \ell_p(E)$ such that $K \subseteq \left\{\sum\limits_{j=1}^\infty \lambda_j x_j : (\lambda_j)_j \in B_{\ell_{p^*}}\right\} = : p\mbox{-conv}\{(x_j)_j\}$. According to \cite{aronruedacontemp}, a polynomial $P \in {\cal P}(^nE;F)$ between Banach spaces is {\it $p$-compact} if $P(B_E)$ is a relatively $p$-compact subset of $F$.  The set ${\cal P}_{{\cal K}_p}(^n E;F)$ of $p$-compact $n$-homogeneous polynomials from $E$ to $F$ is a Banach space with the norm
$$\|P\|_{{\cal K}_p} = \inf\left\{\|(x_j)_j\|_p  : P(B_E) \subseteq p\mbox{-conv}\{(x_j)_j\}\right\}. $$
The linear case $n= 1$ recovers the Banach ideal $({\cal K}_p, \|\cdot\|_{{\cal K}_p})$ of $p$-compact linear operators.


In the lattice environment, we consider the following slightly larger class: An $n$-homogeneous polynomial $P\colon E \to F$ between Banach lattices is said to be {\it solid $p$-compact} if there exists a sequence $(y_j)_j \in \ell_p(F)$ such that $P(B_E) \subseteq \sol{p\mbox{-conv}\{(y_j)_j\}}$. It is easy to check that the collection ${\cal P}_{|\mathcal{K}_p|}(^n E; F)$ of all solid $p$-compact polynomials from $E$ into $F$ is a linear subspace of $\mathcal{P}(^n E; F)$ containing ${\cal P}_{\mathcal{K}_p}(^n E; F)$. For $P \in \mathcal{P}_{|\mathcal{K}_p|}(^n E; F)$, we define
$$ \|P\|_{|{\cal K}_p|} = \inf\left\{\|(x_j)_j\|_p  : P(B_E) \subseteq \sol{p\mbox{-conv}\{(x_j)_j\}}\right\}. $$

Next lemma is a lattice version of \cite[Corollary 3.6(a)]{kim2019} (see also \cite[Lemma 2.2]{kimfenicae}). We recall that a Banach lattice contains no copy of $c_0$ if, and only if, it is a KB-space.

\begin{lemma} \label{lema}
   Let $1 < p < \infty$, $P \in \mathcal{P}(^nE; F)$ and $(y_j)_j \in \ell_p^w(F)$ be given. Then, $P(B_E) \subseteq {\rm sol}({p\mbox{-{\rm conv}}\{ (y_j)_j \}})$ if and only if $|y^\ast|(|P(x)|) \leq \norma{(|y^\ast|(|y_j|))_j}_p$ for all $x \in B_E$ and $y^\ast \in F^\ast$.  For $p=1$, the result holds whenever $E$ contains no copy of $c_0$.
\end{lemma}

\begin{proof} Assume that $P(B_E) \subseteq {\rm sol}({p\mbox{-{\rm conv}}\{ (y_j)_j \}})$ and let $y^\ast \in F^\ast$ and $x \in B_E$ be given. Since $P(x) \in \sol{p\mbox{-conv}\{ (y_j)_j \}}$, there exists $(a_j)_j \in B_{\ell_{p^\ast}}$ such that $|P(x)| \leq \left |\sum\limits_{j=1}^\infty a_j y_j\right | \leq \sum\limits_{j=1}^\infty |a_j|\!\cdot\! |y_j|$. Thus,
    \begin{align*}
        |y^\ast|(|P(x)|) & \leq |y^\ast| \left ( \sum_{j=1}^\infty |a_j| |y_j| \right ) = \sum_{j=1}^\infty |a_j| (|y^\ast|(|y_j|)) \\
        & \leq \norma{(a_j)_j}_{p^\ast}\!\cdot\! \norma{(|y^\ast|(|y_j|))_j}_p \leq \norma{(|y^\ast|(|y_j|))_j}_p.
    \end{align*}
    Conversely, assume that $|y^\ast|(|P(x)|) \leq \norma{(|y^\ast|(|y_j|))_j}_p$ for all $x \in B_E$ and $y^\ast \in F^\ast$ and suppose that there exists $x_0 \in B_E$ such that $P(x_0) \notin \sol{p\mbox{-conv}\{ (y_j)_j \}}$. In particular, $|P(x_0)| \notin (\sol{p\mbox{-conv}\{ (y_j)_j \}})^+$. Letting $A := (\sol{p\mbox{-conv}\{ (y_j)_j \}})^+$, we have: \\
    $\bullet$ $A$ is positive-solid: If $y \in F^+$ and $a \in A$ are such that $y \leq a$, then $y \in \sol{p\mbox{-conv}\{ (y_j)_j \}}$, therefore $y \in A$. \\
    $\bullet$ $A$ is norm bounded: An easy application of H\"older's inequality gives that ${p\mbox{-conv}}\{ (y_j)_j \}$ is norm bounded and so is $A$, as  the solid hull of a norm bounded set is norm bounded as well.
   \\
    $\bullet$ $A$ is weakly closed: First consider $0<p<1$. Since the linear operator $(a_j)_j \in \ell_{p^\ast} \mapsto \sum\limits_{j=1}^\infty a_j y_j \in F$ is weakly continuous, it maps weakly compact sets in $\ell_{p^\ast}$ into weakly compact sets in $F$. Thus, since $B_{\ell_{p^\ast}}$ is weakly compact for $1<p<\infty$, we obtain that ${p\mbox{-conv}}\{ (y_j)_j \}$ is  weakly compact. For $p=1$, weakly compactness of ${p\mbox{-conv}}\{ (y_j)_j \}$  is equivalent to $E$  not containing a copy of $c_0$ (see  \cite[Proposition 2.2]{kimcorrigendum}).   Since the weak topology is sufficiently rich and $p\mbox{-conv}\{ (y_j)_j \}$ is weakly compact (under the assumption of $E$ not containing a copy of $c_0$ if $p=1$), we obtain from \cite[Lemma 19.14]{oikhbergtursi} that $\sol{p\mbox{-conv}\{ (y_j)_j \}}$ is weakly closed. Then $A = \sol{p\mbox{-conv}\{ (y_j)_j \}} \cap F^+$ is also weakly closed.

  By \cite[Proposition 19.7]{oikhbergtursi} there exists $0 \leq y^\ast \in F^\ast$ such that $y^\ast(|P(x_0)|) > \sup\limits_{y \in A} y^\ast(y)$. However,
    \begin{align*}
        y^\ast(|P(x_0)|) & \leq \norma{(y^\ast(|y_j|))_j}_p = \sup_{\varphi \in B_{\ell_{p}^\ast}} |\varphi((y^\ast(|y_j|))_j)| \\
        & = \sup_{(\alpha_i)_i \in B_{\ell_{p^\ast}}} \left | \sum_{j=1}^\infty \alpha_j \, y^\ast(|y_j|) \right |  \leq \sup_{(\alpha_i)_i \in B_{\ell_{p^\ast}}} \sum_{j=1}^\infty y^\ast(|\alpha_j y_j|) \\
       & = \sup_{(\alpha_i)_i \in B_{\ell_{p^\ast}}} y^\ast\left( \sum_{j=1}^\infty |\alpha_j| |y_j|\right) \leq \sup_{y \in A} y^\ast(y).
    \end{align*}
 This contradiction completes the proof.

\end{proof}

\begin{proposition}\label{pcompactprop} Let $1 < p < \infty$. For all Banach lattices $E$ and $F$, $\|\cdot\|_{|{\cal K}_p|}$ is a complete norm on ${\cal P}_{|\mathcal{K}_p|}(^n E; F)$ satisfying conditions {\rm  (I)} and {\rm (II)} of Theorem {\rm \ref{teogeral}}. The following inclusions are continuous:
$${\cal P}_{\mathcal{K}_p}(^n E; F) \subseteq {\cal P}_{|\mathcal{K}_p|}(^n E; F) \subseteq {\cal P}(^n E; F). $$
For $p=1$, the result holds whenever $E$ contains no copy of $c_0$.
\end{proposition}

\begin{proof} We start by checking condition (I) of Theorem {\rm \ref{teogeral}}. Given $P \in {\cal P}_{|\mathcal{K}_p|}(^n E; F)$, let  $(x_j)_j \in \ell_p(F)$ be   such that $ P(B_E) \subseteq \sol{p\mbox{-conv}\{(x_j)_j\}}$. For every $x \in B_E$, there exists $(a_j)_j \in B_{\ell_{p^\ast}}$ such that $|P(x)| \leq \left|\sum\limits_{j=1}^\infty a_j x_j\right|$.  Using that the norm of $F$ is a lattice norm and applying H\"older's inequality, we obtain that $\norma{P(x)} \leq \norma{(x_j)_j}_p$. Taking the supremum over all $x \in B_E$ we get $\|P\| \leq \norma{(x_j)_j}_p$; and taking the infimum over all such sequences $(x_j)_j$ it follows that $\|P\| \leq \|P\|_{|{\cal K}_p|}$. This proves condition (I) and gives the implication $\|P\|_{|{\cal K}_p|} = 0 \Longrightarrow P = 0$.

We skip the easy proof that $\|\lambda P\|_{|{\cal K}_p|} = |\lambda|\!\cdot\! \|P\|_{|{\cal K}_p|}$ for all $P \in  {\cal P}_{|\mathcal{K}_p|}(^n E; F)$ and  $\lambda \in \R$.

Now we shall prove completeness and the triangle inequality using an argument inspired in the proof of \cite[Theorem 2.1]{kimfenicae}. To do so,
let $(P_i)_i$ be a sequence in ${\cal P}_{|\mathcal{K}_p|}(^n E; F)$ such that $\sum\limits_{i=1}^\infty \|P_i\|_{|{\cal K}_p|} < \infty$. We have already proved that $\|P_i\| \leq \|P_i\|_{|{\cal K}_p|}$ holds for every $i \in \N$, so the completeness of $\mathcal{P}( ^n E; F)$ gives a polynomial $P \in \mathcal{P}(^nE; F)$ such that $P = \sum\limits_{i=1}^\infty P_i$ with respect to the usual norm $\|\cdot\|$. Given $\varepsilon > 0$, for each $i \in \N$ there is a sequence $(y_j^i)_j \in \ell_p(F)$ such that $P_i(B_E) \subseteq \sol{p\mbox{-conv}\{ (y_j^i)_j \}}$ and $
\norma{(y_j^i)_j}_p \leq \|P_i\|_{|{\cal K}_p|} + \varepsilon/2^i$. For $i,j \in \N$, we define
 \begin{equation} z_{j}^i = \displaystyle \frac{y_j^i}{(\|P_i\|_{|{\cal K}_p|} + \varepsilon/2^i)^{1/p^{\ast}}} \in F. \label{yjmb}\end{equation}
From
\begin{align*}
    \sum_{i=1}^\infty \sum_{j=1}^\infty \norma{z_j^i}^p & = \sum_{i=1}^\infty \frac{\norma{(y_j^i)_j}_p^p}{(\|P_i\|_{|{\cal K}_p|} + \varepsilon/2^i)^{p/p^{\ast}}}   = \sum_{i=1}^\infty \frac{\norma{(y_j^i)_j}_p^p}{(\|P_i\|_{|{\cal K}_p|} + \varepsilon/2^i)^{p - 1}} \\
    & = \sum_{i=1}^\infty \left (\frac{\norma{(y_j^i)_j}_p}{\|P_i\|_{|{\cal K}_p|} + \varepsilon/2^i} \right )^{p} \left ( \|P_i\|_{|{\cal K}_p|} + \varepsilon/2^i \right ) \\
    & \leq  \sum_{i=1}^\infty (\|P_i\|_{|{\cal K}_p|} + \varepsilon/2^i) = \sum_{i=1}^\infty \|P_i\|_{|{\cal K}_p|} + \varepsilon < \infty,
\end{align*}
it follows that $(z_{j}^i)_{i,j=1}^\infty \in \ell_p(F)$ and $\|(z_{j}^i)_{i,j=1}^\infty\|_p^p \leq  \sum\limits_{i=1}^\infty \|P_i\|_{|{\cal K}_p|} + \varepsilon$.

Let $x \in B_E$ and $y^\ast \in F^\ast$ be given. Since $P_i(B_E) \subseteq \sol{p\mbox{-conv}\{ (y_j^i)_j \}}$ for every $i \in \N$, Lemma \ref{lema} gives $|y^\ast|(|P_i(x)|) \leq \norma{(|y^\ast|(|y_j^i|))_j}_p$, hence
\begin{align*}
    |y^\ast|(|P(x)|) & = |y^\ast|\left ( \left |\sum_{i=1}^\infty P_i(x) \right | \right ) \leq |y^\ast| \left ( \sum_{i=1}^\infty |P_i(x)| \right ) \\
    & = \sum_{i=1}^\infty |y^\ast| (|P_i(x)|) \leq \sum_{i=1}^\infty \norma{(|y^\ast|(|y_j^i|))_j}_p  = \sum_{i=1}^\infty \left ( \sum_{j=1}^\infty [|y^\ast|(|y_j^i|)]^p \right )^{1/p} \\
    & \stackrel{\small\rm(\ref{yjmb})}{=} \sum_{i=1}^\infty (\|P_i\|_{|{\cal K}_p|} + \varepsilon/2^i)^{1/p^{\ast}}\!\cdot\! \left ( \sum_{j=1}^\infty [|y^\ast|(|z_j^i|)]^p \right )^{1/p} \\
    & \leq \sum_{i=1}^\infty (\|P_i\|_{|{\cal K}_p|} + \varepsilon/2^i)^{1/p^{\ast}} \!    \cdot \!\left ( \sum_{j,k=1}^\infty [|y^\ast|(|z_j^k|)]^p\right )^{1/p} \\
    & \leq \left ( \sum_{i=1}^\infty \|P_i\|_{|{\cal K}_p|} + \varepsilon \right )^{1/p^\ast} \!\!\cdot\! \norma{(|y^\ast|(|z_{j}^i|))_{i,j}}_p  = \displaystyle \left \| (|y^\ast|(|w_{j}^i|)) \right \|_p,
\end{align*}
where
$$ w_{j}^i = \displaystyle \left ( \sum_{i=1}^\infty \|P_i\|_{|{\cal K}_p|} + \varepsilon \right )^{1/p^\ast} z_{j}^i $$
for all $i,j \in \N$. Applying Lemma \ref{lema} again  we conclude that $P(B_E) \subseteq \sol{p\mbox{-conv}\{ (w_j^i)_{i,j} \}}$. This proves that $P \in {\cal P}_{|\mathcal{K}_p|}(^n E; F)$ and
$$ \norma{P}_{|\mathcal{K}_p|} \leq \norma{(w_{j}^i)_{i,j}}_p = \left ( \sum_{i=1}^\infty \|P_i\|_{|{\cal K}_p|} + \varepsilon \right )^{1/p^\ast} \!\!\cdot\! \norma{(z_j^i)_{i,j}}_p \leq \left ( \sum_{i=1}^\infty \|P_i\|_{|{\cal K}_p|} + \varepsilon \right )^{1/p^\ast + 1/p}.  $$
Since $1/p + 1/p^\ast = 1$ and $\varepsilon > 0$ is arbitrary, it follows that $\norma{P}_{|\mathcal{K}_p|} \leq \sum\limits_{i=1}^\infty \|P_i\|_{|{\cal K}_p|}$. The triangle inequality is settled. Moreover, the same argument shows that, for every $k \in \N$, $\sum\limits_{i > k} P_i \in {\cal P}_{|\mathcal{K}_p|}(^n E; F)$ and $\left\|\sum\limits_{i > k} P_i\right\|_{|{\cal K}_p|} \leq \sum\limits_{i > k} \norma{P_i}_{|{\cal K}_p|}$. Therefore,
$$ \left\|\sum_{i=1}^k P_i - P\right\|_{|{\cal K}_p|} = \left\|\sum_{i > k} P_i\right\|_{|{\cal K}_p|} \leq \sum_{i > k} \norma{P_i}_{|{\cal K}_p|} \longrightarrow 0 \mbox{ as } k \longrightarrow\infty.$$
Completeness has been established.

As to condition (II), let $P, Q \colon E \to F$ be solid $p$-compact $n$-homogeneous polynomials with $Q \geq 0$ and $|P(x)| \leq Q(|x|)$ for every $x \in E$. 
Let $(y_j)_j \in \ell_p(F)$ be such that $Q(B_E) \subseteq \sol{p\mbox{-conv}\{(y_j)_j\}}$. 
For every $x \in B_E$ there exists a sequence $(a_j)_j \in B_{\ell_{p^\ast}}$ such that 
$$|P(x)| \leq Q(|x|) \leq \left|\sum_{j=1}^\infty a_j y_j\right|.  $$
This proves that $P(x) \in \sol{p\mbox{-conv}\{(y_j)_j\}}$ for every $x \in B_E$, in other words, $P(B_E) \subseteq \sol{p\mbox{-conv}\{(y_j)_j\}}$.  
Hence, $\|P\|_{{|\cal K}_p|} \leq \norma{(y_j)_j}_p$; and taking the infimum over all such sequence $(y_j)$ we get $\|P\|_{{|\cal K}_p|} \leq \|Q\|_{{|\cal K}_p|}$.

The second statement follows from the inequalites $\|\cdot\|\leq \|\cdot\|_{{|\cal K}_p|} \leq \|\cdot\|_{{\cal K}_p}$ (the first was proved above and the second is obvious).
\end{proof}


\begin{proposition} \label{pcompact2}
    Let $P, Q \colon E \to F$ be positive $n$-homogeneous polynomials such that $P \leq Q$. 
    If  $Q$ is solid $p$-compact, then $P$ is solid $p$-compact.
\end{proposition}

\begin{proof} Let $(y_j)_j\in \ell_p(F)$ be  such  that $Q(B_E) \subseteq \sol{p\mbox{-conv}\{(y_j)_j\}}$. In particular we have $\sol{Q(B_E)} \subseteq \sol{p\mbox{-conv}\{(y_j)_j\}}$.
Since
$|P(x)| \leq P(|x|) \leq Q(|x|)$ for every $x \in E$, we have $P(x) \in \sol{Q(B_E)} \subseteq \sol{p\mbox{-conv}\{(y_j)_j\}} $ for every $x \in B_E$. This proves that $P(B_E) \subseteq \sol{p\mbox{-conv}\{(y_j)_j\}}$, therefore $P$ is solid $p$-compact.
\end{proof}

Now we are in the position to apply Theorem \ref{teogeral} to the solid $p$-compact case.
\begin{example}\rm \label{expcompact}
     Let $1 < p < \infty$ and let $E, \, F$ be two Banach lattices with $F$ Dedekind complete.
    By taking $\A = \mathcal{P}_{|\mathcal{K}_p|}(^n E; F)$ and $\norma{\cdot}_{\A} = \|\cdot\|_{|{\cal K}_p|}$ in Theorem \ref{teogeral} we obtain, from Proposition \ref{pcompactprop}, that
    $\mathcal{P}_{|\mathcal{K}_p|}^r(^n E; F)$ is a Banach space with the norm
$$ \|P\|_{|{\cal K}_p|^r} = \inf\left\{\|Q\|_{|{\cal K}_p|}: Q \in \mathcal{P}_{|\mathcal{K}_p|}^+(^n E; F), \, Q \geq |P|\right\}. $$
Proposition \ref{pcompact2} provides the domination property for solid $p$-compact polynomials, so $(\mathcal{P}_{|\mathcal{K}_p|}^r(^n E; F), \|\cdot\|_{||{\cal K}_p|^r})$ is a Banach lattice by Theorem \ref{teogeral}. For $p=1$, the example holds whenever $E$ contains no copy of $c_0$.
\end{example}

\section{Non-embeddability of $c_0$ in $\mathcal{P}^r(^n E; F)$}

Our first task in this section is to give sufficient conditions on a pair of Banach lattices $(E, F)$ in order to obtain that $\A^r$ is a projection band in $\mathcal{P}^r(^n E; F)$ for some closed subspace $\A$ of $\mathcal{P}(^n E; F)$. To establish this result, we will need the following lemma whose proof is contained in the proof of \cite[Theorem 2.9 (iv)$\Rightarrow$(v)]{xanthos}.

\begin{lemma} \label{lemaxanthos}
    If $F$ is an atomic Banach lattice with order continuous norm, then there exists an increasing net of finite rank positive operators $(T_\lambda)_\lambda$ in $ \mathcal{L}(F; F)$ such that $\lim\limits_\lambda T_\lambda (x) = \sup\limits_\lambda T_\lambda (x) = x$ for every $x \in F^+$.
\end{lemma}

A vector-valued map has finite rank if the subspace generated by its range is finite-dimensional.

\begin{theorem} \label{teoremakb}
    Let $E$ and $F$ be two Banach lattices with $F$ Dedekind complete and let $\A$ be a closed subspace of $\mathcal{P}(^n E; F)$. Suppose that $(E, F)$ satisfies the $\mathcal{A}$-domination property and 
    consider the following conditions: \\
    {\rm (1)} $(\mathcal{P}^r(^n E; F), \norma{\cdot}_r)$ contains no copy of $c_0$. \\
    {\rm (2)} $(\A^r, \norma{\cdot}_{{\cal A},r})$ contains no copy of $c_0$. \\
    {\rm (3)} $\A^r$ is a projection band in $\mathcal{P}^r(^n E; F)$. \\
    {\rm (4)} Every positive $n$-homogeneous polynomial $P \colon E \to F$ belongs to $\mathcal{A}$. \\
    Then {\rm (1)}$\Rightarrow${\rm(2)}$\Rightarrow${\rm(3)}. In addition, if $F$ is atomic with order continuous norm and every finite rank positive $n$-homogeneous polynomial from $E$ to $F$ belongs to $\A$, then {\rm(3)}$\Rightarrow${\rm(4)}.
\end{theorem}

\begin{proof}
    We begin noting that it follows from Theorem \ref{teogeral} that $(\mathcal{A}^r, \norma{\cdot}_{\A,r})$ is a Banach lattice with $\norma{P}_{\A,r} = \norma{P}_{r}$ for every $P \in \mathcal{A}^r$. Moreover, $\A^r$ is an ideal in $\mathcal{P}^r(^n E; F)$.

    (1)$\Rightarrow$(2) This follows immediately from the fact that $\mathcal{P}^r(^n E; F)$ contains $\A^r$ as a subspace.

    (2)$\Rightarrow$(3) Assuming that $\A^r$ contains no copy of $c_0$, we get that $\A^r$ is a $KB$-space. To prove that $\A^r$ is a band in $\mathcal{P}^r(^n E; F)$, take a net $(P_\alpha)_\alpha \subset \A^r$ with $0 \leq P_\alpha \uparrow P$ in $\mathcal{P}^r(^n E; F)$. Since $\norma{P_\alpha}_{\A,r} = \norma{P_\alpha}_r \leq \norma{P}_r$ holds for every $\alpha$, we obtain that $(P_\alpha)$ is a norm bounded monotone net in the $KB$-space $\A^r$, hence there exists $Q = \lim_\alpha P_\alpha \in \A^r$. In particular, $\norma{P_\alpha - Q}_r = \norma{P_\alpha - Q}_{\A,r} \to 0$, which implies by \cite[Lemma 2]{mcarthur} that $Q = \sup\limits_\alpha P_\alpha = P$. As $\A^r$ is an ideal in $\mathcal{P}^r(^n E; F)$ (Theorem \ref{teogeral}), this proves that $\A^r$ is a band in $\mathcal{P}^r(^n E; F)$ (see the comment after \cite[Lemma 1.37]{alip}). Since $\mathcal{P}^r(^n E; F)$ Dedekind complete, we obtain from the comment at the beginning of \cite[p.\,36]{alip} that $\A^r$ is a projection band.

      Now we assume that $F$ is atomic with order continuous norm and that every finite rank positive $n$-homogeneous polynomial from $E$ to $F$ belongs to $\A$. To prove that (3)$\Rightarrow$(4), let $P \colon E \to F$ be a positive $n$-homogeneous polynomial. Since $F$ is atomic with order continuous norm, there exists by Lemma \ref{lemaxanthos} an increasing net of finite rank positive operators $(T_\lambda)_\lambda \subset \mathcal{L}^+(F; F)$ with $\lim\limits_\lambda T_\lambda (x) = \sup\limits_\lambda T_\lambda (x) = x$ for every $x \in F^+$. So, $P_\lambda := T_\lambda \circ P \colon E \to F$ defines a net of $n$-homogeneous polynomials such that: \\
    {\rm (i)} For each $\lambda$, $P_\lambda^\otimes = T_\lambda \circ P^\otimes$ is a positive finite rank operator, which yields that $P_\lambda \in \A$. \\
    {\rm (ii)} As $(T_\lambda)_\lambda$ is an increasing net and $P \geq 0$, we obtain that $(P_\lambda)_\lambda$ is an increasing net. Indeed,
    $ P_\mu^\otimes(z) = T_\mu(P^\otimes(z)) \leq T_\lambda(P^\otimes(z)) = P_\lambda^\otimes(z) $
    holds for all $z\in \left (\widehat{\otimes}_{s, |\pi|}^n E \right )^+$ and $\mu \leq \lambda$.\\
    {\rm (iii)} For each $z \in \left (\widehat{\otimes}_{s, |\pi|}^n E \right )^+$,  $ \sup\limits_\lambda P_\lambda^\otimes(z) = \sup\limits_\lambda T_\lambda(P^\otimes(z)) = P^\otimes(z), $ which implies that $\sup\limits_\lambda P_\lambda^\otimes = P^\otimes$, so $\sup\limits_\lambda P_\lambda = P$.

    From the items above, we obtain that $(P_\lambda)_\lambda$ is an increasing net contained in $\A$ such that $\sup\limits_\lambda P_\lambda = P$. Since $\A^r$ is a a projection band in $\mathcal{P}^r(^n E; F)$ by assumption, we conclude that $P \in \A^r$.
\end{proof}

From now on, we shall  establish when the conditions in Theorem \ref{teoremakb} are equivalent whenever $\mathcal{A} = \mathcal{P}_{\mathcal{K}}(^n E; F)$ or $\A = \mathcal{P}_{\mathcal{W}}(^n E; F)$. Recall that the compact linear case was studied by Kalton, the compact polynomial case was considered by P\'erez and the lattice linear case was settled by Xanthos. Our application of Theorem \ref{teoremakb} in the compact case will provide the lattice polynomial case. To do so, we need to recall that a Banach lattice $E$ has the dual positive Schur property if every positive weak* null sequence in $E^\ast$ is norm null (see, e.g., \cite[p.\,760]{wnukdual}).

\begin{theorem} \label{teoprincipal}
    If $E$ is a Banach lattice that fails the dual positive Schur property and $F$ is an infinite dimensional atomic Banach lattice with order continuous norm, then the following are equivalent for every $n \in \N$:\\
    {\rm (1)} $(\mathcal{P}^r(^n E; F), \norma{\cdot}_r)$ contains no copy of $c_0$.
    \\
    {\rm (2)} $(\mathcal{P}_{\mathcal{K}}^r(^n E; F), \norma{\cdot}_{{\cal K},r})$ contains no copy of $c_0$.
    \\
    {\rm (3)} $\mathcal{P}_{\mathcal{K}}^r(^n E; F)$ is a projection band in $\mathcal{P}^r(^n E; F)$. \\
    {\rm (4)} Every positive $n$-homogeneous polynomial from $E$ to $F$  is compact. 
\end{theorem}

\begin{proof} Since $F$ is atomic with order continuous norm, in particular $F$ is Dedekind complete, we have from Example \ref{corcompacto} that $(\mathcal{P}_{\mathcal{K}}^r(^n E; F), \norma{\cdot}_{{\cal K},r})$ is a Banach lattice such that $\norma{P}_{{\cal K},r} = \norma{P}_r$ for every $P \in \mathcal{P}_{\mathcal{K}}^r(^n E; F)$. As finite rank polynomials are compact, the implications (1)$\Rightarrow$(2)$\Rightarrow$(3)$\Rightarrow$(4) follow from Theorem \ref{teoremakb}.

    (4)$\Rightarrow$(1) We prove first that  $\left (\widehat{\otimes}_{s, |\pi|}^n E \right )^\ast$ has order continuous norm norm.  If $\left (\widehat{\otimes}_{s, |\pi|}^n E \right )^\ast$ fails to have order continuous norm, then $\widehat{\otimes}_{s, |\pi|}^n E $ contains a sublattice isomorphic to $\ell_1$ (see \cite[Theorem 2.4.14]{meyer})
    which is the range of a positive projection (see \cite[Proposition 2.3.11]{meyer}), i.e.
    there exist a sequence $(z_k)_k \subset \widehat{\otimes}_{s, |\pi|}^n E $ which is equivalent to the canonical basis $(e_k)_k$ of $\ell_1$ and a positive operator $\pi \colon \widehat{\otimes}_{s, |\pi|}^n E  \to \ell_1$ such that $\pi(z_k) = e_k$ for every $k \in \N$.
     On the other hand, since $F$ is infinite dimensional, there exists  a positive bounded sequence $(y_k)_k$ in $F$ with no convergent subsequence. Defining $S \colon \ell_1 \to F$ by $S((a_j)_j) = \sum_{i=1}^\infty a_i y_i$, we have that $S \geq 0$, consequently $T = S \circ \pi \colon \widehat{\otimes}_{s, |\pi|}^n E \to F$ is also a positive operator which is not compact, because $T(z_k) = S(e_k) = y_k$ for every $k \in \N$ and $(y_k)_k$ has no convergent
     subsequence in $F$. Since $F$ is atomic with order continuous norm, we get from \cite[Theorem 4.1]{libu} that the positive $n$-homogeneous polynomial $P \colon E \to F$ whose linearization is $T$ is not compact, which yields a contradiction. Thus $\left (\widehat{\otimes}_{s, |\pi|}^n E \right )^\ast$ has order continuous norm. Moreover, as
     $F$ also has order continuous norm, we obtain from \cite[Theorem 2.8]{chenwick} that $(\mathcal{K}^r(\widehat{\otimes}_{s, |\pi|}^n E; F), \norma{\cdot}_k)$ has order continuous norm. Now, considering the identification obtained in Example \ref{corcompacto}, the identification from \cite[Proposition 3.4]{bubuskes} and the assumption, we have that
     $$ \mathcal{K}^r(\widehat{\otimes}_{s, |\pi|}^n E; F) \cong \mathcal{P}_{\mathcal{K}}^r(^n E; F) =\mathcal{P}^r(^n E; F) \cong \mathcal{L}^r(\widehat{\otimes}_{s, |\pi|}^n E; F),$$
     consequently $(\mathcal{L}^r(\widehat{\otimes}_{s, |\pi|}^n E; F), \norma{\cdot}_r)$ has order continuous norm.

     Assume, for sake of contradiction, that $F$ is not a $KB$-space. So, there exists $(y_k)_k \subset F^+$ equivalent to the canonical basis of $c_0$. On the other hand, since $E$ fails the dual positive Schur property, there exists a weak* null sequence $(z_k^\ast)_k \subset (E^\ast)^+$ that is not norm null. By passing to a subsequence if necessary, we may assume that $\norma{z_k^\ast} \geq \varepsilon$ for every $k \in \N$. Letting $x_k^\ast = \frac{z_k^\ast}{\norma{z_k^\ast}}$, $k \in \N$, we obtain that $(x_k^\ast)_k \subset S_{E^*}^+$ is a weak* null sequence.
    As $((x_k^\ast(x))^n)_k \in c_0$ for every $x \in E$ and $(y_k)_k$ is equivalent to the canonical basis of $c_0$, the series $\sum\limits_{k=1}^\infty (x_k^\ast(x))^n y_k$ converges in $F$ for every $x \in E$, so we can define $P \colon E \to F$ by $P(x) = \sum\limits_{k=1}^\infty (x_k^\ast(x))^n y_k$ which is a positive $n$-homogeneous polynomial. Next we prove that $(P^\otimes)^\ast$ is not a compact operator, which will give a contradiction. Indeed, if $(y_k^\ast)_k$ is the sequence of biorthogonal functionals associated to $(y_k)_k$, then $(y_k^\ast)_k$ is a bounded sequence such that
    $$ (P^\otimes)^\ast(y_k^\ast) (\otimes^n x) = y_k^\ast(P(x)) = \sum_{i=1}^\infty (x_i^\ast(x))^n y_k^\ast (y_i) = (x_k^\ast (x))^n = (x_k^\ast)^n (\otimes^n x)$$
    for every $x \in E$, which yields that $(P^\otimes)^\ast(y_k^\ast) = (x_k^\ast)^n$ in $\left (\widehat{\otimes}_{s, |\pi|}^n E \right )^\ast \cong \mathcal{P}^r(^n E)$. Since $x_k^\ast \in S_{E^\ast}^+$ for every $k \in \N$, we have that $(P^\otimes)^\ast(y_k^\ast)$ is a positive polynomial, consequently
    $$ \norma{(P^\otimes)^\ast(y_k^\ast)}_r = \norma{(P^\otimes)^\ast(y_k^\ast)} = \norma{(x_k^\ast)^n} = \sup_{x \in B_E} |x_k^\ast(x)|^n = \left ( \sup_{x \in B_E} |x_k^\ast(x)| \right )^n = 1. $$
    On the other hand, since $x_k^\ast \cvfe 0$ we obtain that $(P^\otimes)^\ast(y_k^\ast)$ is a normalized sequence which is pointwise convergent to zero, hence it cannot have a convergent subsequence. So, $(P^\otimes)^\ast$ cannot be a compact operator. Thus $P^\otimes$ is not a compact operator either, and by \cite[Theorem 4.1]{libu} we obtain that $P$ is not a compact polynomial, which contradicts the hypothesis. Therefore $F$ is a $KB$-space, which implies by \cite[Theorem 4]{chenfeng} that $(\mathcal{L}^r(\widehat{\otimes}_{s, |\pi|}^n E; F), \norma{\cdot}_r)$ is a $KB$-space, and we are done.
\end{proof}

The following is an immediate consequence of Theorem \ref{teoprincipal}.

\begin{corollary} \label{cor1}
    If $F$ is an infinite dimensional atomic Banach lattice with order continuous norm and every positive $n$-homogeneous polynomial $P \colon E \to F$  is compact, then $\mathcal{P}^r(^nE; F)$ has order continuous norm. In addition, if $E$ fails the dual positive Schur property, then $\mathcal{P}^r(^nE; F)$ is a $KB$-space.
\end{corollary}

Let us remark what happens for almost limited polynomials.

\begin{remark}\rm
    Let $E$ and $F$ be two Banach lattices with $F$ being atomic with order continuous norm. It follows from \cite[Theorem 4.2(3)]{machelbour} that every almost limited linear operator from any Banach space to $F$ is compact. So, by considering the identifications obtained in Example \ref{corcompacto} and Example \ref{coralmostlimited}, we get that $\mathcal{P}_{\mathcal{K}}^r(^n E; F) = \mathcal{P}_{\rm al}^r(^n E; F)$. Consequently, whenever $E$ fails the dual positive Schur property, the following conditions are all equivalent by Theorem \ref{teoprincipal}: \\
    {\rm (1)} $(\mathcal{P}^r(^n E; F), \norma{\cdot}_r)$ contains no copy of $c_0$.
    \\
    {\rm (2)} $(\mathcal{P}_{\rm al}^r(^n E; F), \norma{\cdot}_{{\rm al},r})$ contains no copy of $c_0$.
    \\
    {\rm (3)} $\mathcal{P}_{\rm al}^r(^n E; F)$ is a projection band in $\mathcal{P}^r(^n E; F)$. \\
    {\rm (4)} Every positive $n$-homogeneous polynomial from $E$ to $F$ is almost limited. 
\end{remark}

      To see when the conditions of Theorem \ref{teoremakb} are all equivalent in the weakly compact case, we need to recall that a Banach lattice $E$ is said to have the positive Grothendieck property if every positive weak* null sequence in $E^\ast$ is weakly null (see, e.g., \cite[p.\,760]{wnukdual}).

\begin{theorem} \label{teoprincipal2}
    Let $n \in \N$, let $E$ be a Banach lattice that fails the positive Grothendieck property and let $F$ be an atomic Dedekind complete Banach lattice such that $\mathcal{P}^r(^n E; F)$ has order continuous norm. Then, the following are equivalent: \\
    {\rm (1)} $(\mathcal{P}^r(^n E; F), \norma{\cdot}_r)$ contains no copy of $c_0$.    \\
    {\rm (2)} $(\mathcal{P}_{\mathcal{W}}^r(^n E; F), \norma{\cdot}_{{\cal W},r})$ contains no copy of $c_0$.
    \\
    {\rm (3)} $\mathcal{P}_{\mathcal{W}}^r(^n E; F)$ is a projection band in $\mathcal{P}^r(^n E; F)$. \\
    {\rm (4)} Every positive $n$-homogeneous polynomial from $E$ to $F$ is weakly compact.
\end{theorem}

\begin{proof}
     Note first that, since $\mathcal{P}^r(^n E; F) \cong \mathcal{L}^r(\widehat{\otimes}_{s, |\pi|}^n E; F)$ has order continuous norm, $F$ also has order continuous norm (see \cite[Proposition 1]{chenfeng}). Thus, it follows from Example \ref{corfraccompacto} that $(\mathcal{P}_{\mathcal{W}}^r(^n E; F), \norma{\cdot}_{{\cal W},r})$ is a Banach lattice such that $\norma{P}_{{\cal W},r} = \norma{P}_r$ for every $P \in \mathcal{P}_{\mathcal{W}}^r(^n E; F)$. Using that finite rank polynomials are weakly compact, we notice that the implications (1)$\Rightarrow$(2)$\Rightarrow$(3)$\Rightarrow$(4) follow from Theorem \ref{teoremakb}.




    (4) $\Rightarrow$ (1) 
    To conclude that $\mathcal{P}^r(^n E; F) \cong \mathcal{L}^r(\widehat{\otimes}_{s, |\pi|}^n E; F)$ 
    is a $KB$-space, by \cite[Theorem 4]{chenfeng} it is enough to check that $F$ is a $KB$-space. Assume, for sake of contradiction, that $F$ is not a $KB$-Space. So, there exists $(y_k)_k \subset F^+$ equivalent to the canonical basis of $c_0$.
    On the other hand, since $E$ fails the positive Grothendieck property, there exists a positive weak* null sequence $(x_k^\ast)_k \subset E^\ast$ which is not weakly null. So, there exists $x^{\ast\ast} \in E^{\ast \ast}$ such that $\lim\limits_k x^{\ast \ast} (x_k^\ast) \neq 0$. By passing to a subsequence if necessary, we may assume that $|x^{\ast \ast} (x_k^\ast)| > \varepsilon$ holds for every $k \in \N$ and some $\varepsilon > 0$. As in the proof of Theorem \ref{teoprincipal}, $P(x) = \sum\limits_{k=1}^\infty (x_k^\ast(x))^n y_k$, $x \in E$, defines a positive $n$-homogeneous polynomial from $E$ to $F$ such that $(P^\otimes)^\ast(y_k^\ast) = (x_k^\ast)^n$ in $\left (\widehat{\otimes}_{s, |\pi|}^n E \right )^\ast \cong \mathcal{P}^r(^n E)$, where $(y_k^\ast)_k$ is the sequence of biorthogonal functionals associated to $(y_k)_k$.  Since each $P_k := (x_k^\ast)^n$ is weakly continuous on bounded sets and $|\widetilde{P_k}(x^{\ast \ast})| = |x^{\ast \ast} (x_k^\ast)|^n \geq \varepsilon^n$ holds for every $k \in \N$, where $\widetilde{P_k} \colon E^{\ast \ast} \to \R$ is the Aron-Berner extension of $P_k$ (see \cite{aronberner, daviegamelin}), we get from \cite[Theorem 3.3]{buquaterly} that $(P_k)_k$ has no weakly null subsequence in $\mathcal{P}^r(^n E)$. Futhermore, since $x_k^\ast \cvfe 0$ in $E^\ast$, we have that $(P_k)_k$ is pointwise convergent to zero, hence it cannot have a weakly convergent subsequence. This yields that $(P^\otimes)^\ast$ cannot be a weakly compact operator. Thus $P^\otimes$ is not a weakly compact (see \cite[Theorem 5.23]{alip}). By \cite[Theorem 3.2]{libu} we obtain that $P$ is  not  weakly compact polynomial, a contradiction. Therefore $F$ is a $KB$-space, which implies that $(\mathcal{L}^r( \widehat{\otimes}_{s, |\pi|}^n E; F), \norma{\cdot}_r)$ is a $KB$-space, and we are done.
\end{proof}

It is well known that $KB$-spaces have order continuous norms. The next results give conditions for lattices of polynomials with order continuous norms to be $KB$-spaces. First we give the result for vector-valued polynomials, which follows immediately from Theorem \ref{teoprincipal2}.

\begin{corollary} \label{cor2}
Let $n \in \N$ and let $E, F$ be two Banach lattices with $F$  Dedekind complete such that $\mathcal{P}^r(^n E; F)$ has order continuous norm. If $E$ fails the positive Grothendieck property, $F$ is atomic and every positive $n$-homogeneous polynomial $P \colon E \to F$ is weakly compact, then $\mathcal{P}^r(^n E; F)$ is a KB-space.
\end{corollary}

The case of scalar-valued polynomials reads as follows:

\begin{corollary} \label{cor3}
    Let $n \in \N$ and let $E$ be a Banach lattice such that $\mathcal{P}^r(^n E)$ has order continuous norm. If $E$ fails the positive Grothendieck property, then $\mathcal{P}^r(^n E)$ is a KB-space.
\end{corollary}

\begin{proof} Just take $F = \R$ in the corollary above, use that $\R$ is atomic and that every $n$-homogeneous polynomial $P \colon E \to \R$ is weakly compact.
\end{proof}

\noindent G. Botelho  \hspace*{13.1em}V. C. C. Miranda\\
Faculdade de Matem\'atica \hspace*{6.6em}  Centro de Matem\'atica, Computa\c c\~ao e Cogni\c c\~ao \\
Universidade Federal de Uberl\^andia \hspace*{2em} Universidade Federal do ABC \\
38.400-902 -- Uberl\^andia -- Brazil \hspace*{3.4em} 09.210-580 -- Santo Andr\'e -- 		Brazil.  \\
e-mail: botelho@ufu.br \hspace*{7.8em} e-mail: colferaiv@gmail.com

\medskip

\noindent P. Rueda\\
Departamento de An\'alisis Matem\'atico\\
Universidad de Valencia\\
46.100 -- Burjasot, Valencia -- Spain\\
e-mail: pilar.rueda@uv.es

\end{document}